\documentclass[12pt]{amsart}
\usepackage{graphicx}
\DeclareGraphicsExtensions{.pdf,.eps,.fig,.pdf_tex}
\usepackage{color}
\usepackage{amsmath}
\usepackage{amssymb}
\usepackage{amsthm}
\usepackage{amscd}
\usepackage{amsfonts}
\usepackage{float}
\usepackage{enumerate}
\usepackage{lineno, blindtext}

\usepackage{subcaption}

\usepackage{hyperref}
\usepackage{fancyhdr}
\usepackage{mathtools}

\theoremstyle{plain} \numberwithin{equation}{section}
\theoremstyle{definition}

\newtheorem{theorem}{Theorem}[section]

\newtheorem{lem}[theorem]{Lemma}
\newtheorem{proposition}[theorem]{Proposition}
\newtheorem{definition}[theorem]{Definition}

\newtheorem{remark}[theorem]{Remark}
\newtheorem{example}[theorem]{Example}
\newtheorem{cor}[theorem]{Corollary}

\newcommand{\frall}{\:\forall\:}

\usepackage{comment}
\newtheorem{thm}{Theorem}

\theoremstyle{definition}

\title{Classification and Ideal Lattices of Leavitt Path Algebras}
\author{Seth Yoo, Yvan Grinspan}

\date{}

\begin{document}

\begin{abstract}
    Leavitt path algebras are free algebras subject to relations induced by directed graphs. This paper investigates the ideals of Leavitt path algebras, with an emphasis on the relationship between graph-theoretic properties of a directed graph and the ideals of the associated Leavitt path algebra. We begin by presenting a new proof of a fundamental result characterizing graded and non-graded ideals of a Leavitt path algebra using a condition on the number of closed paths at each vertex in its directed graph. Appealing to this result, we then classify the Leavitt path algebras of directed graphs with two vertices up to isomorphism and determine all possible lattice structures of a class of well-behaved ideals possessed by such algebras.
\end{abstract}

\maketitle

\section*{\textbf{Introduction}}

Leavitt path algebras are a class of non-commutative algebras induced by directed graphs that have close ties to $C^*$-algebras, $K$-theory, dynamics, and more \cite{ara2009k}\cite{arnone2023graded}\cite{ruiz2013ideal}. In this paper, we restrict our focus to Leavitt path algebras induced by directed graphs with finitely many vertices and edges between them, with our focus being on the connections between properties of directed graphs and the ideals of their induced Leavitt path algebras. Of particular importance will be the relationship between closed paths on graphs and a natural grading, according to path length, with which a Leavitt path algebra can be endowed. This relationship is presented by Abrams, Ara, and Molina \cite{abrams2017leavitt}\cite{abrams2005leavitt}, and it serves as the basis of our main results.  

The remainder of this paper is organized as follows.

In Section 1, we introduce general definitions of algebras, their ideals, concluding with a discussion of algebras defined by generators and relations. 

In Section 2, we define Leavitt path algebras, establish some basic results about them, and introduce the aforementioned grading. From there, we define graded ideals, which are the central object of study in this paper. 

In Section 3, we state and prove our main results. We begin with a theorem connecting a graph's properties with the ideals of its Leavitt path algebra (Theorem 2). Then, using this result to define a new class of ideals called $\lambda$-reducible ideals, we conclude our results by studying Leavitt path algebras induced by graphs with two vertices (Theorem 3 and Theorem 4). 

Finally, in Section 4, we indicate potential directions for further investigation. 

\section{\textbf{Algebras: Basic Definitions}}

\begin{definition}
    Let $K$ be a field and $A$ a vector space over $K$. $A$ is an \textbf{algebra} (also called a \textbf{$K$-algebra}) if endowed with an additional operation $A\times A\rightarrow A$, denoted by $\cdot$, such that the following hold for all $x,y,z\in A$ and $a,b\in K$:
    \begin{itemize}
        \item[$i$)] $(x\cdot y)\cdot z = x\cdot (y\cdot z)$ (Associativity)
        \item[$ii$)] $(x + y)\cdot z = x\cdot z + y\cdot z$ and $z\cdot (x + y) = z\cdot x + z\cdot y$ (Distributivity)
        \item[$iii$)] $(ax)\cdot (by) = (ab)(x\cdot y)$ (Compatibility with Scalars)
    \end{itemize}
\end{definition}
In other words, $K$-algebras are vector spaces over $K$ endowed with multiplication. Alternatively, $K$-algebras can also be viewed as rings endowed with scalar multiplication over $K$. For ease of notation, $x\cdot y$ will often be written as $xy$. 

\begin{definition}
    A subset $B\subseteq A$ is a \textbf{subalgebra} of $A$ if $B$ is an algebra under the operations defined on $A$. 
\end{definition}

\begin{definition}
    An \textbf{ideal} of an algebra $A$ is a subalgebra $I$ such that, for any $a\in A$ and any $r\in I$, $ar\in I$ and $ra \in I$. 
\end{definition}

\begin{definition}
    Let $X$ be a nonempty set. We refer to finite sequences of elements of $X$ as \textbf{words} over $X$. Then, the \textbf{free algebra} on $X$ over $K$ is the module over $K$ generated by words over $X$, endowed with multiplication defined as concatenation, meaning words are products of their elements. 
\end{definition}

Free algebras can be alternatively described as generalizations of polynomial rings where the variables, $X$, do not commute. Note that in this paper we only consider free algebras over fields, so the free algebras we will be discussing are vector spaces, rather than merely modules, endowed with concatenation as multiplication.

\begin{example} 
    Consider the free algebra $A$ over $K$ on the set $\{a,b\}$, or in other words, the vector space generated by words over $\{a,b\}$: $\{a, b, aa, bb, ab, ba, \ldots \}$. In this example, $a$ and $b$ are said to be \textbf{generators} of $A$, and $A$ is alternatively denoted by $K\langle a,b\rangle$. \\
    Now, consider the following algebra: $K\langle a,b \mid ab - ba = 0\rangle$. This is the algebra generated by $\{a,b\}$ such that $ab - ba = 0$, which implies $ab = ba$. In other words, this is the algebra generated by $\{a,b\}$ where $a$ and $b$ commute, and it easily follows that multiplication is commutative in general. As a vector space, the bases $ab$ and $ba$ are now identical, whereas those bases in $K\langle a,b\rangle$ are not. Here, $ab-ba$ is said to be a \textbf{relation} that defines $K\langle a,b \mid ab - ba = 0\rangle$. \\

    Observe that $K\langle a,b \mid ab - ba = 0\rangle\cong K[a,b]$, the polynomial ring in two variables with coefficients in $K$. Additionally, $K\langle a,b \mid ab - ba = 0, ab = 1\rangle\cong K[x,x^{-1}]$, the Laurent polynomials. 
\end{example} 

    This process of defining relations on free algebras is more rigorously described by taking the free algebra and considering its quotient by the ideal generated by the relation. 

\section{\textbf{Leavitt Path Algebras}}

In this section, we will define Leavitt path algebras along with other graph theoretic and algebraic terms that are relevant to our main results. 

\section*{Definitions and Basic Results}

\begin{definition} A \textbf{directed graph} $E = (E^0, E^1, r, s)$ consists of a set $E^0$ of \textbf{vertices}, a set $E^1$ of \textbf{edges}, and two functions $r,s : E^1\rightarrow E^0$. For a given edge $e \in E^1$, $s(e) = v$ and $r(e) = w$ if $e$ goes from vertex $v$ to vertex $w$. That is, $s$ and $r$ map edges to their \textbf{source} and \textbf{range} vertices respectively.

In this paper, we will assume that $E$ has both a finite number of vertices and edges; namely, we assume $E$ is a \textbf{finite graph}. Furthermore, the term \textbf{graph} will always refer to directed graphs. 
\end{definition}

\begin{definition} Let $E = (E^0, E^1, r, s)$ be a graph and consider subsets $F^0\subseteq E^0$ and $F^1 \subseteq E^1$ such that when $r$ and $s$ are restricted to $F^1$, we have that $r(F^1),s(F^1)\subseteq F^0$. We call $F = (F^0, F^1, r|_{F^1}, s|_{F^1})$ a \textbf{subgraph} of $E$. 
\end{definition}

The subgraph of a graph $E$ can simply be viewed as any graph contained within $E$. 

\begin{definition}
    Let $E$ be a directed graph. Define the \textbf{set of ghost edges of $E$} to be the set $(E^1)^* := \{e^* \mid e\in E^1\}$, where $e^*$ is an edge such that $r(e^*) = s(e)$ and $s(e^*) = r(e)$. 
\end{definition}

More colloquially, for each edge of a directed graph, its corresponding ghost edge is an edge between the same two vertices as the regular edge that points in the opposite direction.

\begin{definition} The \textbf{Leavitt path algebra} of a graph $E = (E^0, E^1, r, s)$, denoted by $L_K(E)$ or $L(E)$, is the free K-algebra generated by $E^0 \cup E^1 \cup (E^1)^*$\ subject to the relations: \\ 

($V$) $\frall v, w \in E^0, vw =
    \begin{cases}
        0 &\text{if } w\neq v, \\
        v &\text{if } w = v
    \end{cases}$ \\

($E_1$) $\frall e \in E^1, er(e) = s(e)e = e$\\

($E_2$) $\frall e \in E^1, e^*s(e)= r(e)e^* = e^*$ \\

($CK_1$) $\frall e, f \in E^1, e^*f =
    \begin{cases}
        0 &\text{if } e\neq f, \\
        r(e) &\text{if } e = f
    \end{cases}$ \\

($CK_2$): $\frall v \in E^0, v = \sum_{e \in E^1, s(e)=v} ee^*$.  
\end{definition}

Given that we are only considering finite graphs, we have that any $L(E)$ is unital with $1 = \sum_{v\in E^0} v$, which follows from $(V), (E_1)$, and $(E_2)$. \\

While most Leavitt path algebras are quite complicated, some familiar algebras arise naturally as Leavitt path algebras of specific graphs. For example, the Leavitt path algebra of the rose with one petal is isomorphic to the Laurent polynomials over $K$, and the Leavitt path algebra of the oriented n-line is isomorphic to the algebra of $n\times n$ $K$-matrices.

\begin{figure}[H]
    \centering
    \includegraphics[width=0.6\linewidth]{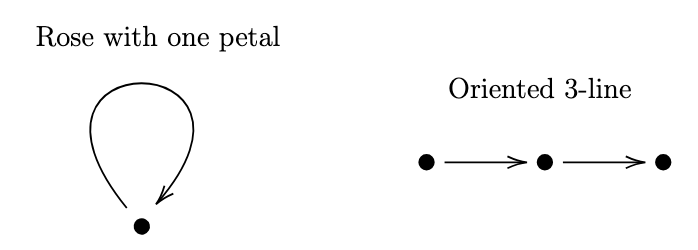}
    \caption{Graphs with familiar Leavitt path algebras}
    \label{fig1}
\end{figure}

\begin{definition}
    A \textbf{path} in $L(E)$ is a product of edges $e_1\dots e_n$ such that $r(e_i) = s(e_{i+1})$ for $i = 1,\dots n$. The set of all paths is denoted $Path(E)$. 
\end{definition}

\begin{definition}
    A \textbf{closed path} in $L(E)$ is a path $e_1\dots e_n$ such that $s(e_1) = r(e_n)$ and $n \geq 1$. 
\end{definition}

\begin{definition}
    A \textbf{closed simple path} in $L(E)$ is a closed path $e_1\dots e_n$ such that $s(e_i)\neq s(e_1)$ for $i = 2,\dots ,n$. Defining $s(e_1) = v$, we say $e_1\dots e_n$ is a closed simple path based at $v$. 
\end{definition}

\begin{definition}
    A \textbf{cycle} in $L(E)$ is a closed simple path $e_1\dots e_n$ such that $s(e_i)\neq s(e_j)$ for every $i\neq j$. 
\end{definition}

\begin{definition}
    The number of edges in an element $\mu = e_1\dots e_n\in Path(E)$ is called the \textbf{degree} of $\mu$. In this case, we say $deg(\mu) = n$. For any $v\in E^0$, we define $deg(v) = 0$. We can extend the domain of $s$ and $r$ to include all paths by defining $s(\mu) = s(e_1)$ and $r(\mu) = r(e_n)$. We define $\mu^* = e_n^* e_{n-1}^* \ldots e_1^*$. We will also extend the notion of $Path(E)$ to the set of all paths involving both edges and ghost edges; let us denote this extension by $Path(\hat{E})$
\end{definition}

\begin{definition}
   A \textbf{loop} in $L(E)$ is an edge $e$ where $s(e) = r(e)$. A loop is a cycle of degree 1. 
\end{definition} 

\begin{definition}
    A graph $E = (E^0, E^1, r, s)$ satisfies \textbf{Condition (K)} if for every $v \in E^0$, there are either zero or at least two closed simple paths based at $v$. 
\end{definition}

\begin{proposition}\label{2.12}
    Given a graph $E$, its Leavitt path algebra can be expressed as $L_K(E) = \text{span}_K(\{\alpha\beta^* \mid \alpha,\beta \in Path(E), r(\alpha) = r(\beta) \}).$
\end{proposition}
    This result follows from $(CK_1).$ One important observation is that, while elements of the form $\alpha\beta^*$ span $L_K(E)$, a single element is not necessarily uniquely represented by an element of that form. 

\begin{example}
    Consider the simple example $L(E)$, where $E$ is the digraph with one vertex $v$ and one edge $e$ such that $r(e) = s(e) = v$. Then, we can express the element $v$ as $v = e^*e = e^*e^*e^*eee,$ etc. However, if we have an element of the form, say $ee^*e^*e^*e$, it can be expressed as $e^*$ by applications of $(CK1)$ and $(CK2)$. 
\end{example}

\begin{definition}
    Consider all representations $\alpha_i\beta_i^*$ of a given monomial $\mu\in L(E)$, and consider the set $S$ of all degrees of the $\beta_i$'s, namely $S = \{deg(\beta_i) \mid \alpha_i\beta_i^* = \mu\}$. We call the minimal element of $S$ the \textbf{degree in ghost edges} of $\mu$, denoted $gdeg(\mu)$. 
\end{definition}

\begin{definition}
    Let $X\subseteq E^0$ be a set of vertices in a directed graph $E$. $X$ is said to be \textbf{hereditary} if for any vertices $v$ and $w$, whenever $v\in X$ and there exists $\mu\in Path(E)$ such that $s(\mu) = v$ and $r(\mu) = w$, then $w\in X$. 
\end{definition}

\begin{definition}
    Let $X\subseteq E^0$ be a set of vertices in a directed graph $E$. $X$ is said to be \textbf{saturated} if, for any $v\in E^0$ such that $s^{-1}(v) \neq \varnothing$, $\{r(e) \mid s(e) = v\}\subseteq X \implies v\in X$. 
\end{definition}

Given a set $X$ of vertices, we call the smallest hereditary and saturated set containing $X$ the \textbf{hereditary saturated closure} of $X$, and denote it $T(X)$.

\section*{\texorpdfstring{$\mathbb{Z}$-Grading and Ideals}{Z-Grading and Ideals}}

Having introduced the definition of a Leavitt path algebra along with some graph theoretic terms which we will show related to these algebras, we now turn to the ideals of these algebras. These ideals, in particular distinctions surround a Leavitt path algebra's graded and non-graded ideals, will be central objects of study in the remainder of this paper. 

\begin{definition}
    A $K$-algebra is said to be \textbf{$\mathbb{Z}$-graded} if it can be broken up into a set of $K$-vector subspaces $\{A_i \mid i\in\mathbb{Z}\}$ such that the following hold:
    \begin{itemize}
        \item [($i$)] $A = \bigoplus_{i\in\mathbb{Z}}A_i$, and
        \item [($ii$)] $a_ia_j\in A_{i+j}$ for $a_i\in A_i, a_j\in A_j$ with $i,j\in\mathbb{Z}$.
    \end{itemize}
    Each $A_i$ in a $\mathbb{Z}$-graded algebra is called a \textbf{homogeneous component of degree $i$}. An element $a_i\in A_i$ is called a \textbf{homogeneous element of degree $i$}. 
\end{definition}

\begin{definition}
    An ideal $I$ in a $\mathbb{Z}$-graded $K$-algebra $A$ is a \textbf{graded ideal} if for any element $y \in I$, which can be decomposed into $y = \sum_{i\in\mathbb{Z}}y_i$ where each $y_i\in A_i$, we have $y_i\in I$ for every $i\in\mathbb{Z}$. 
\end{definition}

\begin{proposition}[Abrams, Ara, Molina]
    Let $E$ be a directed graph and $L(E)$ its associated Leavitt path algebra. Then, $L(E)$ is $\mathbb{Z}$-graded with a grading induced by path length. That is, $L(E) = \bigoplus_{i\in\mathbb{Z}}A_i$, where $x = \alpha\beta^* \in A_i$ if $deg(\alpha) - deg(\beta) = i$.
\end{proposition}

\begin{proposition}
    Any ideal $I$ of a $\mathbb{Z}$-graded algebra $A$ generated solely by homogeneous elements is a graded ideal.
\end{proposition}
\begin{proof}
    Let $I$ be some ideal generated by homogeneous elements. Then, an arbitrary element of $I$ will be of the form $x\alpha y$, where $x,y\in A$ and $\alpha$ is a generator of $I$ and therefore homogeneous. These $x,y\in A$ are going to be sums of the form $x = \sum_{i=1}^n x_i$ and $y = \sum_{j=1}^m y_j$, where the $x_i$'s and $y_j$'s are monomials. We can expand $x\alpha y = (\sum_{i=1}^n x_i)\alpha(\sum_{j=1}^m y_j)$, and as $\alpha$ is a generator for $I$, we have that $x_i\alpha y_j \in I$ for each $i = 1,\dots ,n$ and $j = 1,\dots ,m$. Furthermore, as $\alpha$ is homogeneous and the $x_i$'s and $y_j$'s are monomials, we know that $x_i\alpha y_j$ is homogeneous as well. From this, it follows that all of the homogeneous components of $x\alpha y$ are also in $I$, completing our proof.
\end{proof}

\begin{proposition} Let $I \subseteq L(E)$ be an ideal. Then, $I \cap E^0$ is a hereditary and saturated set of vertices in E. 
\end{proposition}
\begin{proof}
    Consider any $v\in E^0$ such that $v\in I$ and consider also any $\alpha \in Path(E)$ such that $s(\alpha) = v$. Notice that $(\alpha ^*)v(\alpha) = w \in E^0$ must be an element of $I$ (by repeated applications of $(CK_1)$ and the fact that $I$ is an ideal). With this, we have shown that, given a vertex $v\in I$, if there exists a path from $v$ to any other vertex $w\in I$, it must be that $w$ is also in $I$. In other words, $I\cap E^0$ is hereditary.
    
    Now, consider any $v\in E^0$ such that \{$r(e) \mid e\in E^1$ and $s(e) = v$\} $\subseteq I\cap E^0$. Then, $er(e)e^* = ee^*\in I$ for all $e$ such that $r(e) \in$ \{$r(e) \mid e\in E^1$ and $s(e) = v$\} (because $I$ is an ideal). Hence, it follows that $v = \sum_{s(e)=v} ee^* \in I$ (by $(CK_2)$), or in other words, $v\in I\cap E^0$, which means $I\cap E^0$ is saturated, as desired.
\end{proof}

\begin{thm} [Abrams, Ara, Molina (Reduction Theorem)]
    Let $E$ be an arbitrary graph and $K$ any field. For any nonzero element $\alpha\in L_K(E)$ there exist $\mu,\eta\in Path(E)$ such that either:
    \begin{itemize}
        \item [($i$)] $0\neq\mu^*\alpha\eta = kv$ for some $k\in K, k\neq0$, and $v\in E^0$, or
        \item [($ii$)] $0\neq\mu^*\alpha\eta = p(c)$, where $c$ is a cycle without exits and $p(x)$ is a nonzero polynomial in $K[x,x^{-1}]$.
    \end{itemize}
\end{thm}

The result of this theorem was our inspiration for many of our subsequent results.

\section{\textbf{Main Results}}

\section*{Graded and Non-Graded Ideals: A Graph Theoretic Approach}
\begin{lem}\label{3.1}
    Let $v$ be a vertex such that there is exactly one closed simple path based at $v$, and let that closed simple path be denoted $\lambda_1 = e_1\dots e_n$. Then, we have the following:
    \begin{itemize}
        \item [($i$)] $\lambda_1$ is a cycle.
        \item [($ii$)] For every $s(e_i)$ with $i = 1,\dots ,n$, there is exactly one closed simple path, which is also a cycle, based at $s(e_i)$, namely $\lambda_i = e_i\dots e_ne_1\dots e_{i-1}$. 
        \item[($iii$)] All closed paths based at any $s(e_i)$ are of the form $\lambda_i^k$, $k > 0$, $\lambda_i^0 = s(e_i)$.
    \end{itemize}
\end{lem}
\begin{proof}
    To show ($i$), suppose that $\lambda_1$ is not a cycle, namely, that there exists $i\neq j$ such that $s(e_i) = s(e_j)$. Without loss of generality, suppose $i<j$. Then, notice that both $\lambda_1 = e_1\dots e_n$ and $e_1\dots e_{i-1}e_j\dots e_n$ are closed simple paths based at $v$, a contradiction. Hence, it must be that $\lambda_1$ is a cycle. \\
    Now we prove ($ii$). Suppose that for some $s(e_i)$, $i = 2,\dots ,n$, there exists closed simple path $\mu$ based at $s(e_i)$ such that $\mu\neq\lambda_i$. Then, writing $\mu = f_1\dots f_m$, we have two cases to address. Either $r(f_k) = v$ for some $k = 1,\dots, m$, or $r(f_k) \neq v$ for all such $k$. In the first case, $e_1 \dots e_{i - 1} f_1 \dots f_k$ is a closed simple path based at $v$ not equal to $\lambda_1$, and in the second case, $e_1\dots e_{i-1}\mu e_i\dots e_n$ is such a closed simple path. Regardless, we arrive at a contradiction. Hence, it follows that the only closed simple path based at any $s(e_i)$ is $\lambda_i$. The claim that $\lambda_i$ is a cycle follows immediately from ($i$). With this, ($iii$) now follows from the fact that all closed paths are products of closed simple paths.
\end{proof}

\begin{thm}\label{2}
  Let $E$ be a finite directed graph and $L(E)$ its associated Leavitt path algebra over a field K. All ideals of $L(E)$ are graded if and only if $E$ satisfies Condition (K).
\end{thm}
\begin{proof} We begin by proving that $L(E)$ has only graded ideals if $E$ satisfies Condition (K). To this end, let $I\subseteq L(E)$ be an arbitrary ideal, and let $\alpha\in I$ be an arbitrary element of $I$. From Proposition \ref{2.12}, we know $\alpha$ can be expressed in the form $\sum_{i=1}^k c_i\alpha_i\beta_i^*$ for some $k\geq1$, where $c_i\in K^\times$, $\alpha_i,\beta_i\in Path(E)$, and $r(\alpha_i) = r(\beta_i)$ $\forall i\in \{1,\dots ,k\}$. Assume $\alpha_i\beta_i^* \neq \alpha_j\beta_j^*$ for all $j\neq i$ and that each $\alpha_i\beta_i^*$ is written in its simplest form (i.e. that all combinations of ghost edges and regular edges that can be reduced to vertices by $(CK_1)$ have been). Lastly, let us order the terms in $\alpha$ such that $deg(\beta_i)\geq deg(\beta_{i+1})$ for $i=1,\dots ,k-1$. In other words, $\alpha$ will be ordered such that its terms are of descending degree in ghost edges. Using this form, we can multiply $\alpha$ by a succession of elements to show that $I$ is graded; that is, if $\alpha\in I$, then writing $\alpha = \sum_{k \in\mathbb{Z}} \alpha_k$ using the direct sum decomposition given by the path-length grading, we obtain that $\alpha_k \in I$ for all $k$. \\

If $\alpha$ is homogeneous, then $\alpha$ satisfies the desired condition. Therefore, we consider the case where $\alpha$ is non-homogeneous. We start by multiplying $\alpha$ by $\beta_1$ from the right. Doing so yields $\alpha\beta_1 = c_1\alpha_1 + \sum_{i=2}^mc_i\alpha_i\gamma_i$, where $m\leq k$ (since some terms might reduce to 0 via $(CK_1)$), and $r(\alpha_1) = r(\gamma_i) = r(\beta_1)$ $\forall i = 2,\dots ,m$. Because the terms in $\alpha$ were ordered above by descending ghost-edge degree, $\beta_1$ is as least as long as $\beta_i$. By $(CK_1)$, any term in which $\beta_i^*$ is not reduced to $r(\beta_i)$ through multiplication by $\beta_1$ reduces to 0, so in each remaining term, it must be that $\beta_i\gamma_i = \beta_1$. \\

Now, $\alpha\beta_1$ is a linear combination of elements $\nu_i := \alpha_i\gamma_i\in Path(E)$ that share the range $r(\beta_1)$. We can arrange the terms of $\alpha\beta_1$ in increasing degree order, i.e. $\sum_{j=1}^m c_j\nu_j$ where $deg(\nu_j)\leq deg(\nu_{j+1})$ for $j = 1,\dots ,m-1$. Now, we multiply $\alpha\beta_1$ by $\nu_1^*$ from the left, yielding $\nu_1^*\alpha\beta_1 = c_1r(\nu_1) + \sum_{j=2}^n c_j\mu_j$, where $n\leq m$ and $\nu_j = \nu_1 \mu_j \frall j = 1,\dots ,n$.\footnote{A caveat in notation here is that the $\nu_1$ and $\beta_1$ do not, in general, correspond to the same term of $\alpha$ because terms are removed and rearranged between the assignment of the subscript of $\beta_1$ and that of $\nu_1$.} The last statement here is forced by $(CK_1)$ and the increasing-degree order of the terms in the same way as the relationship between remaining terms in $\alpha\beta_1$ above. Vitally, this operation mandates that every term afterwards ends in a closed loop $\mu_j$: each $\nu_j$ has range $r(\beta_1)$, but simultaneously begins with $\nu_1$, which has the same range, meaning the addendum, which we call $\mu_j$, must be a closed loop. Therefore, we have $s(\mu_j) = r(\nu_j) = r(\beta_1) = r(\mu_j)$ $\forall j = 1,\ldots,n$, so that $\nu_1^*\alpha\beta_1$ is a linear combination of $r(\beta_1)$ and closed paths based at it. \\

As Condition (K) holds on $E$, each vertex $v\in E^0$ has 0 or at least 2 closed simple paths based at $v$. In this case, if there are no closed simple paths, and therefore no closed paths at all, based at $r(\beta_1)$, then $\nu_1^*\alpha\beta_1$ is simply $c_1r(\beta_1)$, so $ r(\beta_1)\in I$.
\\

Now suppose instead that there are at least 2 closed simple paths based at $r(\beta_1)$ and that $\nu_1^*\alpha\beta_1$ includes a sum of closed paths based at $r(\beta_1)$, as described above. Taking two arbitrary closed simple paths based at $r(\beta_1)$, we call them $\eta_1$ and $\eta_2$. Using these two paths, we can reduce $\nu_1^*\alpha\beta_1$ through a recursive process of multiplying by $\eta_1^*$ or $\eta_2^*$ from the left and $\eta_1$ or $\eta_2$ from the right. \\

If multiplying $\nu_1^*\alpha\beta_1 = c_1r(\beta_1) + \sum_{j=2}^n c_j\mu_j$ by $\eta_1^*$ from the left does not reduce any terms in the sum to zero via $(CK_1)$, then each $\mu_j$ begins with $\eta_1$, so multiplying by $\eta_2^* \neq \eta_1^*$ from the left eliminates every term besides $c_1r(\beta_1)$, and we have that $\eta_2^*\nu_1^*\alpha\beta_1 = c_1\eta_2^*$. Thus, multiplying from the right by $c_1^{-1}\eta_2$ yields $\eta_2^*\nu_1^*\alpha\beta_1c_1^{-1}\eta_2 = c_1\eta_2^*c_1^{-1}\eta_2 = r(\beta_1) \in I$.
\\

Alternatively, suppose that multiplying $\nu_1^*\alpha\beta_1$ by $\eta_1^*$ from the left does eliminate some terms in the sum. In this case, we have $\eta_1^*\nu_1^*\alpha\beta_1 = c_1\eta_1^* + \sum_{i=1}^q c_i\tau_i$, where $q < n$. Then, multiplying from the right by $\eta_1$ yields $\eta_1^*\nu_1^*\alpha\beta_1\eta_1 = c_1r(\beta_1) + \sum_{i=1}^qc_i\lambda_i$, where $\lambda_i$ is a closed path based at $r(\beta_1)$ $\forall i = 1,\dots ,q$, and we have returned to the same initial conditions but with strictly fewer summands. It follows that by repeatedly multiplying by $\eta_1^*$ on the left and $\eta_1$ on the right, followed by $\eta_2^*$ on the left and $\eta_2$ on the right, $\nu_1^*\alpha\beta_1$ will eventually be reduced to $c_1r(\beta_1)$. \\

In all cases, we have that $r(\beta_1)\in I$, so that $c_1\alpha_1r(\beta_1)\beta_1^* = c_1\alpha_1\beta_1^*\in I$. Hence, subtracting $c_1\alpha_1\beta_1^*$ from $\alpha$ yields another element of $I$; notice that this new element is simply $\alpha$ without its first term. We can therefore repeat the entire process described thus far with $\alpha - c_1\alpha_1\beta_1^*$, showing that $c_2\alpha_2\beta_2^*\in I$, and consider $\alpha - c_1\alpha_1\beta_1^* - c_2\alpha_2\beta_2^*\in I$. Continuing in this fashion, we can show that every summand in $\alpha$ is in $I$, as desired.\\

To prove the converse, we prove that if $E$ does not satisfy Condition (K), then $L(E)$ has a non-graded ideal. To this end, supposing $E$ does not satisfy Condition (K), there exists a vertex $v\in E^0$ such that there is exactly one closed simple path based at $v$. Let this closed simple path be $\lambda$, where $r(\lambda) = s(\lambda) = v$, and consider the ideal generated by $v+\lambda$. We claim that this ideal is non-graded.
First, we observe that such an ideal consists of all elements of the form $\delta_1(v+\lambda)\delta_2$, where $\delta_1,\delta_2\in L(E)$. \\

To prove that $\langle v+\lambda\rangle$ is a non-graded ideal, it is sufficient to prove that $v$ is not in $\langle v+\lambda\rangle$. That is, we prove that no $\delta_1,\delta_2\in L(E)$ satisfy $\delta_1 (v+\lambda) \delta_2 = v$. Notice that it is sufficient to consider solely the cases where $\delta_1,\delta_2$ are monomials because otherwise, the expression can be decomposed into a sum of expressions where this is true, and this sum is equal to $v$ if and only if one term, resulting from multiplication by a specific pair of monomials in $\delta_1$ and $\delta_2$, is $v$, and the rest are zero. Let us rewrite $\delta_1 = c_1\alpha_1\beta_1^*$ and $\delta_2 = c_2\alpha_2\beta_2^*$, where $c_1,c_2\in K\setminus\{0\}$, $\alpha_1,\alpha_2,\beta_1,\beta_2\in Path(E)$, $r(\alpha_1) = r(\beta_1)$, and $r(\alpha_2) = r(\beta_2)$. First, notice that our choice of $\delta_1(v+\lambda)\delta_2 = c_1c_2\alpha_1\beta_1^*(v+\lambda)\alpha_2\beta_2^* = c\alpha_1\beta_1^*(v+\lambda)\alpha_2\beta_2^*$ reduces to $0$ unless $s(\beta_1) = s(\alpha_2) = v$, so we narrow down our investigation to cases where this holds. Furthermore, notice that $c\alpha_1\beta_1^*(v+\lambda)\alpha_2\beta_2^*$ is a vertex only if $\beta_1^*(v+\lambda)\alpha_2$ reduces to a vertex, but the uniqueness of the closed simple path based at $v$ implies that is impossible. Namely, in order for $\beta_1^*(v+\lambda)\alpha_2$ not to vanish, we need $\beta_1$ and $\alpha_2$ to be closed paths based at $v$. By Lemma \ref{3.1}, all closed paths based at $v$ are of the form $\lambda^k$ for $k\geq0$, where $\lambda^0 = v$. In this case, we see that neither $\beta_1^*v\alpha_2$ nor $\beta_1^*\lambda\alpha_2$ reduce to 0 via $(CK_1)$, and $deg(\beta_1^*\lambda\alpha_2) - deg(\beta_1^*v\alpha_2) = n$. This means that $\beta_1^*(v+\lambda)\alpha_2$ is always a non-homogeneous element, which implies $\langle v+\lambda\rangle$ is a non-graded ideal, completing the proof. \end{proof} 

\begin{cor}
    There is a bijection between graded ideals and (the set of) hereditary saturated sets of vertices.
\end{cor}
\begin{proof}
    The set of vertices in any ideal (and thus any graded ideal) is hereditary and saturated by Proposition 2.28, and the reduction done in the proof of Theorem 2 demonstrates that any graded ideal is generated by a set of vertices, as all non-graded generators can be multiplied into vertices. Meanwhile, two distinct hereditary saturated sets of vertices cannot generate the same ideal because each set of generating vertices is also the set of vertices contained in the resulting ideal. 
\end{proof}

\begin{definition}
   Let $v$ be a vertex such that there is exactly one cycle based at $v$ (we will call such a vertex a \textbf{K1 vertex}), and let $\lambda$ be this cycle (we will call cycles of this type \textbf{K1 cycles}). Consider some Laurent polynomial (excluding monomials) made up of powers of $\lambda$, where $\lambda^0 = v$ and $\lambda^{-1}=\lambda^*$ (e.g. $\lambda^2 + \lambda + (\lambda^*)^3$). Notice that any such polynomials can be multiplied by $\lambda$'s (from the right) or $\lambda^*$'s (from the left) such that the polynomial has $\lambda^0=v$ as its term with lowest degree. We call such polynomials with lowest degree term $v$ \textbf{cycle polynomials} or \textbf{$\lambda$-polynomials}. We will also denote these polynomials by $p(\lambda)$. 
\end{definition} 

\begin{definition}
    Let $\lambda = e_1\dots e_n$ be the K1 cycle associated with some $\lambda$-polynomial. Consider the set of all edges $f \notin \{e_1,\ldots,e_n\}$ such that $s(f) = s(e_j)$ for some $j \in\{1,\ldots,n\}$. Then, $\bigcup_{f\in\{g\in E^1 \mid s(g) = s(e_j)\}}\{r(f)\}$ is the \textbf{exit range of $\lambda$}.
\end{definition} 

Note that for such an edge $f$, we have that $r(f) \notin \{r(e_1),\ldots,r(e_n)\}$ because otherwise, the condition in the definition of a $\lambda$-polynomial that $\lambda$ is the only cycle based at $v$ would be violated. Similarly, the smallest hereditary set containing $\bigcup_{f\in\{g\in E^1 \mid s(g) = s(e_j)\}}\{r(f)\}$ does not contain any of the vertices in $\{r(e_1),\ldots,r(e_n)\}$.

An ideal generated by a non-homogenous sum of powers of a closed simple path based at a vertex such that there exists exactly one closed simple path based at $v$ is equivalent to an ideal generated by a $\lambda$-polynomial. Namely, it is equivalent to the ideal generated by the $\lambda$-polynomial that results from multiplying the sum by appropriate $\lambda$'s or $\lambda^*$'s. With this in mind, we obtain: 

\begin{cor}
    Any ideal with a $\lambda$-polynomial as one of its generators contains the (graded) ideal generated by the exit range of $\lambda$.
\end{cor}
\begin{proof}
   Suppose $I\subseteq L(E)$ is an ideal with some $p(\lambda)$ as one of its generators, where $\lambda^0 = v$. Without loss of generality, we can proceed as if every exit on the cycle is at $v$. To see this, suppose $\lambda = e_1\dots e_n$ and that an exit was at $r(e_i)$. Then, $(e_1\dots e_i)^*p(\lambda)(e_1\dots e_i) = p(e_{i+1}\dots e_ne_1\dots e_i)\in I$. We then simply relabel $r(e_i) = v$ and $e_{i+1}\dots e_ne_1\dots e_i = \lambda$. In other words, returning to our originally fixed $v$, we see that the ideal $I = \langle \dots ,p(\lambda),\dots \rangle$ is equal to the ideal $\langle \dots ,p(e_{i+1}\dots e_ne_1\dots e_i),\dots \rangle$ for any $e_i$.\\
   
   Now, for any exit, we have an edge, which we name $f$, such that $s(f) = v$ and $f\neq e_i$ for any $1\leq i\leq n$. Choose some $k \in \mathbb{Z}$ such that $p(\lambda)$ has a non-zero $\lambda^{-k}$ term, or equivalently, $p(\lambda)\lambda^k$ contains a term $cv$ where $c \in K \setminus \{0\}$. It can be quickly seen that $f^*p(\lambda)\lambda^k f = r(f)$, so we have reduced $p(\lambda)$ to a vertex $r(f)$ outside of $\lambda$. It follows immediately that every vertex in the exit range of $\lambda$ is in $I$, hence the ideal generated by those vertices is contained in $I$, as desired.
\end{proof}

\begin{definition}
    We say an ideal is \textbf{$\lambda$-reducible} if it is generated by a set of the form $$\{p_a(\lambda_a)\}_{a\in A} \cup \{v_j\}_{j\in\{1,\dots ,n\}},$$ 
    where $A$ is an arbitrary indexing set, the $p_a$'s are polynomials, the $\lambda_a$'s are K1 cycles (allowing for repetition), and $\{v_1,\dots ,v_n\}$ is a hereditary, saturated, possibly empty set of vertices.
\end{definition}

In the definition above, we choose to restrict our attention to generating sets with polynomials of K1 cycles because if we allowed for a cycle $\gamma$ based at a vertex $v$ with either zero or two cycles, the reduction process described in the proof of Theorem 2 would allow us to replace any $p(\gamma)$ with $v$. We could then take the hereditary saturated closure of $\{v\} \cup \{v_j\}_{j\in\{1,\dots ,n\}}$ as the vertex portion of our generating set and remove all polynomials made up of powers of cycles based at $v$, including $\gamma$, without changing the generated ideal. 

This definition is reasonable to consider given the results we have because it captures all graded ideals (for which the polynomial set is empty), and it interacts well with the reduction methods we have developed, which will allow us to reduce containments of ideals to divisibility of their generators. To this end, we outline a process below through which we can obtain a canonical generating set for any $\lambda$-reducible ideal. 

Suppose $I$ is a $\lambda$-reducible ideal. Then, by Lemma 3.1, we can view the cycles $\lambda_a$ as cycles disregarding their basepoints because these basepoints can be changed by multiplying by parts of the cycle appropriately. Under this view, we can bundle our family $\{ p_a (\lambda_a) \}_{a\in A}$ into a union $\cup_{i=1}^m \{p_{a_i}(\lambda_i)\}_{a_i\in A_i}$, where $A_i$ is an indexing set for the collection of polynomials of a given cycle $\lambda_i$, and there are a finite number $m$ of these cycles because our graphs are assumed to be finite. 
    
Consider the collection of polynomials $\{p_{a_i}(\lambda_i)\}_{a_i \in A_i}$ of one cycle $\lambda_i$ in this generating set of $I$. Because $K[x]$ is a principal ideal domain, there exists a unique polynomial (up to scalar multiplication) $p_i(x)\in K[x]$ such that each $p_{\alpha_i}(x) = q(x)p_i(x)$ for some $q(x)\in K[x]$ and $p_i(x)$ is some linear combination (viewing $K[x]$ as a $K[x]$-module) of polynomials in $\{p_{a_i}(x)\}_{a_i\in A_i}$. From this, for each $i\in\{1,\dots ,m\}$, we can replace the family $\{p_{a_i}(\lambda_i)\}_{a_i\in A}$ by the singleton $\{p_i(\lambda_i)\}$. Making these replacements, our generating set reduces to $$\{p_i(\lambda_i)\}_{i\in\{1,\dots ,m\}} \cup \{v_j\}_{j\in\{1,\dots ,n\}}.$$ In this simplification process, a given $p_i(\lambda_i)$ becomes the vertex at which $\lambda_i$ is based exactly when the initial generating set contains coprime polynomials of $\lambda_i$. Lastly, as our choices of $p_i$ are unique up to scalar multiplication, the requirement that our polynomials be monic yields a finite list of unique generating polynomials. To obtain our canonical generating set, we can then take our set of vertices to be $I\cap E^0$ and remove any polynomials from our list above made up of powers of cycles based at any of the vertices in $I\cap E^0$. Note that $I\cap E^0$ is the hereditary saturated closure of the union $\{v_j\}_{j\in\{1,\dots ,n\}} \cup \{w_i\}_{i\in\{1,\dots ,k\}}$, where the $w_i$'s are the new vertices added from the case when $p_i(\lambda_i)$ is a vertex.\\

To summarize, given a $\lambda$-reducible ideal $I$, we can express it uniquely as an ideal generated by a set of the form $$\{p_i(\lambda_i)\}_{i\in\{1,\dots ,k\}} \cup (I\cap E^0),$$ where each $p_i$ is monic and the $\lambda_i$'s are distinct K1 cycles not based at any of the vertices in $I\cap E^0$. 

\begin{definition}
    Given a $\lambda$-reducible ideal $I$, let us refer to its unique generating set of the form described above as its \textbf{$\lambda$-reduction}. We will denote this set by $\Lambda(I)$.
\end{definition}

\begin{example}
    Consider the following graph, which we will call $E$, and consider the $\lambda$-reducible ideal $I\subseteq L_{K}(E)$ generated by $\{e + v, e - v, f + w\}$.

    \begin{figure}[H]
        \centering
        \includegraphics[width=0.5\linewidth]{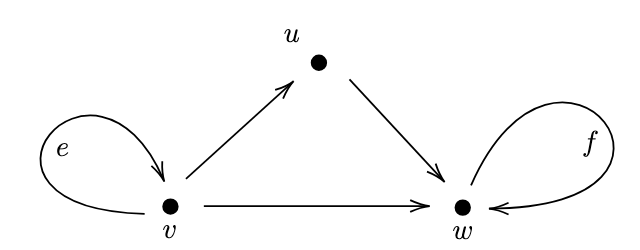}
        \caption{Example 3.8}
        \label{fig2}
    \end{figure}
    
    We will compute $\Lambda(I)$. First, we group our polynomials by cycle to obtain the rewriting of our generating set as $\{e + v, e - v\} \cup \{f + w\}$. Then, taking the greatest common divisors of these sets, we see that $\{e + v, e - v\}$ reduces to $\{v\}$ while $\{f + w\}$ remains the same. Because we have obtained a new vertex $v$ during our reduction of the polynomials, we can take the hereditary saturated closure of $\{v\}$, which yields that $I\cap E^0 = \{u,v,w\}$. Then, as $w\in I$, we have $I = \langle u, v, w, f+w \rangle = \langle u, v, w \rangle$, concluding that $\Lambda(I) = \{u,v,w\}$.

\end{example}

\begin{remark}
    A valuable result we now have is that the $\lambda$-reducible ideals are in bijection with the set of generating sets of the form $\Lambda(I)$. That is, it suffices to consider every ideal generated by a hereditary and saturated set of vertices and a set of monic polynomials of distinct K1 cycles not based at any of the vertices in the hereditary and saturated set of vertices.
\end{remark}

Having simplified our $\lambda$-reducible ideals to ideals generated by sets of this form, we obtain the following proposition: 

\begin{proposition}
    Let $I$ and $I'$ be two non-graded ideals of a given Leavitt path algebra such that their $\lambda$-reductions contain identical sets of vertices. Let us write $\Lambda(I) = \{p_i(\lambda_i)\}_{i\in\{1,\dots ,m\}} \cup \{v_j\}_{j\in\{1,\dots ,n\}}$ and $\Lambda(I') = \{p_i'(\lambda_i)\}_{i\in\{1,\dots ,m\}} \cup \{v_j\}_{j\in\{1,\dots ,n\}}$, where $\{1,\ldots,m\}$ indexes all of the K1 cycles in the graph, meaning some of the $p_i$'s and $p_i'$'s may be the zero polynomial. Then $I \subseteq I'$ if and only if $p_i'$ divides $p_i$ for all $i\in\{1,\dots ,m\}$. 
\end{proposition}
\begin{proof}
    If $p_i'$ divides $p_i$ for all $i\in\{1,\dots ,m\}$, it follows easily that $I \subseteq I'$. Conversely, suppose that $I\subseteq I'$ but $p_i' \nmid p_i$ for some $i\in\{1,\dots ,m\}$. Then, as $p_i(\lambda_i) \in I \subseteq I'$, both $p_i(\lambda_i)$ and $p_i'(\lambda_i)$ are elements of $I'$. Viewing $p_i$ and $p_i'$ as elements of $K[x]$, it follows that $q(\lambda_i) \in I'$, where $q(x) \in K[x]$ is the greatest common divisor of $p_i$ and $p_i'$, which we assume to be monic. Therefore, we have that $q(\lambda_i) \in I'$, that $q$ divides $p_i'$, and that $q \neq p_i'$. However, by the construction of our $\lambda$-reductions, no such $q(\lambda_i)\in I'$ can exist, so we have a contradiction.  
\end{proof}

\section*{Directed Graphs with Two Vertices}

Having developed some results about the ideals of Leavitt path algebras, we will apply our insights to study the case of Leavitt path algebras induced by graphs with two vertices. 

\begin{remark}
    There is a 1-1 correspondence between Leavitt path algebras with finite generating sets and finite directed graphs, up to isomorphism. 
\end{remark}

\begin{thm}
    Given 2 vertices and k edges, there are $\frac{n(n+1)(3k-4n+1)}{3}+(n+1)\lceil\frac{k+1}{2}\rceil$ Leavitt path algebras up to isomorphism where $n=\lceil\frac{k}{2}\rceil$.
\end{thm}
\begin{proof}
     As Leavitt path algebras are uniquely associated to directed graphs, this result is equivalent to the following, which we will prove: given 2 vertices and k edges, there are $\frac{n(n+1)(3k-4n+1)}{3}+(n+1)\lceil\frac{k+1}{2}\rceil$ graphs up to isomorphism, where $n=\lceil\frac{k}{2}\rceil$. To prove this, we begin by fixing the number of loops on each vertex (call them $u$ and $v$) and counting the possible arrangements of edges between them. Furthermore, by symmetry, we fix the number of loops on $u$ to be greater than or equal to the number of loops on $v$. With this in mind, we can start counting the possibilities given some $k$. 
     
    If there are 0 loops on $u$, then $\lceil\frac{k+1}{2}\rceil$ graphs are added. Namely, each of the $k$ vertices can go from $u$ to $v$ or $v$ to $u$, yielding $k+1$ possibilities, but symmetry requires those possibilities to be divided by 2. However, if $k$ is even, then $\frac{k+1}{2}\notin\mathbb{Z}$. Considering more concretely what takes place with an even number of edges, we see that there are $\frac{k}{2} + 1$ different graphs. These graphs are constructed by starting with the graph with all edges facing one direction, then flipping one edge in the opposite direction at a time until half of the edges are facing the opposite direction. Noticing that for odd $k$ we have $\frac{k+1}{2}$ different graphs and for even $k$ we have $\frac{k}{2} + 1 = \frac{k+2}{2}$ different graphs, we see that we have $\lceil\frac{k+1}{2}\rceil$, as claimed. 
    
    If there is 1 loop on $u$, then we have $k$ possible graphs where there are 0 loops on $v$ made from flipping the $k-1$ edges between $u$ and $v$ one at a time, and by a similar argument as the case with 0 loops on $u$, if we have one loop on $u$ and $v$, we get $\lceil\frac{k-1}{2}\rceil$. Hence, we have $k+\lceil\frac{k-1}{2}\rceil$ different graphs.
    
    More generally, we see that given $m$ loops on $u$, we have $(k-m+1) + (k-m) + (k-m-1) +\dots + (k-m-(m-1)) + \lceil\frac{k-(2m-1)}{2}\rceil$ different graphs, where each term adds all graphs with one more loop added to $v$ and the final ceiling term adds the graphs where $v$ also has $m$ loops.

    The sum of all of these possible graphs can now be expressed as the sum of three different sums:
    \begin{itemize}
        \item [($i$)] $k + 2(k-2) + 3(k-4) + \dots $
        \item [($ii$)] $(k-1) + 2(k-3) + 3(k-5) + \dots $
        \item [($iii$)] $\lceil\frac{k+1}{2}\rceil + \lceil\frac{k-1}{2}\rceil + \lceil\frac{k-3}{2}\rceil + \dots $
    \end{itemize}
    However, notice that for fixed $k$, these sums need to be terminated such that negative terms are not counted. So, we must define the sums in such a way that $(k-c)$ terms are not added for any $c>k$. \\
    We begin by combining sums ($i$) and ($ii$). Rearranging terms, we get that together they form the sum $2(1+2+3+\dots )k - ((0+1) + 2(2+3) + 3(4+5) + \dots )$, which has general term $2nk - n(4n-3)$. Notice that this general term breaks into the two terms $n(k - (2n-2))$ and $n(k-(2n-1))$, terms from ($i$) and ($ii$), respectively. We want to terminate this series at an $n$ such that all positive $(k-(2n-2))$ and $(k-(2n-1))$ are included while no negatives are. In other words, we want to find the appropriate $n$ for the series $\sum_{i=1}^n (2ik - i(4i-3))$. \\
    
    \textbf{Claim:} $n = \lceil\frac{k}{2}\rceil$. \\

    To demonstrate this claim, we consider 2 cases: (1) $k$ is even, (2) $k$ is odd. \\
    First, suppose $k$ is even. Then, $n = \frac{k}{2}$, which implies $k = 2n$. Now, we consider the last terms in the sum, namely $n(k-(2n-2))$ and $n(k-(2n-1))$. We see that $n(k-(2n-2)) = 2n\geq 0$ and $n(k-(2n-1)) = n\geq 0$. If we added another term to the series, we would have the terms $(n+1)(k-(2(n+1)-2)) = 0$ and $(n+1)(k-(2(n+1)-1)) = -(n+1) < 0$, which we clearly do not want to include. Hence, for even $k$, we see that $n = \lceil\frac{k}{2}\rceil$ is the appropriate stopping point. \\
    Now, suppose $k$ is odd. Then, $n = \frac{k+1}{2}$, which implies $k = 2n-1$. Once again, we consider the last terms, namely $n(k-(2n-2))$ and $n(k-(2n-1))$. We see that $n(k-(2n-2)) = n\geq 0$ and $n(k-(2n-1)) = 0\geq 0$. If we added another term to the series, we would have the terms $(n+1)(k-(2(n+1)-2)) = -(n+1) < 0$ and $(n+1)(k-(2(n+1)-1)) = -2(n+1) < 0$, which we do not want to include. Hence, for odd $k$, we see that $n = \lceil\frac{k}{2}\rceil$ is the appropriate stopping point as well, proving our claim. \\

    To summarize, combining sums ($i$) and ($ii$) and accounting for when they should stop, we have $\sum_{i=1}^n (2ik-i(4i-3))$, where $n = \lceil\frac{k}{2}\rceil$. Notice that starting the sum at $i=0$ has no effect on the sum, so we will rewrite our sum as $\sum_{i=0}^n 2ik-i(4i-3)$, where $n = \lceil\frac{k}{2}\rceil$. \\
    Now, we consider ($iii$), namely the sum $\lceil\frac{k+1}{2}\rceil + \lceil\frac{k-1}{2}\rceil + \lceil\frac{k-3}{2}\rceil + \dots $. This sum has general term $\lceil\frac{k-(2n-1)}{2}\rceil$, so we have the sum $\sum_{i=0}^n\lceil\frac{k-(2i-1)}{2}\rceil$. Once again, we need to find the appropriate $n$ such that no negative terms are included. \\

    \textbf{Claim:} $n = \lceil\frac{k}{2}\rceil$. \\

    To see this, we study the cases when $k$ is even and $k$ is odd once again. Supposing $k$ is even, we have $n = \frac{k}{2}$, which implies $k = 2n$. With this, we see that $\lceil\frac{k-(2n-1)}{2}\rceil = \lceil\frac{1}{2}\rceil = 1\geq 0$, while the next term $\lceil\frac{k-(2(n+1)-1)}{2}\rceil = \lceil\frac{-1}{2}\rceil = 0$. Seeing that the term after our stopping point is 0 and that subsequent terms would be negative, we see that $n = \lceil\frac{k}{2}\rceil$ is an appropriate stopping point for even $k$. Supposing $k$ is odd, we have that $n = \frac{k+1}{2}$, which implies $k = 2n-1$. Considering the last term in the sum, With this, we see that $\lceil\frac{k-(2n-1)}{2}\rceil = \lceil\frac{0}{2}\rceil = 0$, where any subsequent terms would be less than or equal to 0. With this, we see that $n = \lceil\frac{k}{2}\rceil$ is an appropriate stopping point for odd $k$, as desired. \\

    Hence, we now have that ($iii$) with the appropriate stopping point is $\sum_{i=0}^n\lceil\frac{k-(2i-1)}{2}\rceil$ where $n = \lceil\frac{k}{2}\rceil$. \\
    Noticing that both this sum and the combined sum of ($i$) and ($ii$) can be combined, we get the sum $\sum_{i=0}^n (2ik - i(4i-3) + \lceil\frac{k-(2i-1)}{2}\rceil)$ where $n = \lceil\frac{k}{2}\rceil$. We now manipulate this sum, which yields:

    \begin{align*}
        \sum_{i=0}^n \bigg(2ik - i(4i-3) + \bigg\lceil\frac{k-(2i-1)}{2}\bigg\rceil\bigg) &= \sum_{i=0}^n \bigg((2k+3)i - 4i^2 + \bigg\lceil\frac{k+1-2i}{2}\bigg\rceil\bigg)\\
        &= \frac{n(n+1)(2k+3)}{2} + \frac{2n(n+1)(2n+1)}{3}\\
        &\quad + \sum_{i=0}^n\bigg(\bigg\lceil\frac{k+1}{2}\bigg\rceil - i\bigg)\\
        &= \frac{n(n+1)(6k-8n+5)}{6} - \frac{n(n+1)}{2}\\
        &\quad + (n+1)\bigg\lceil\frac{k+1}{2}\bigg\rceil\\
        &= \frac{n(n+1)(3k-4n+1)}{3}+(n+1)\bigg\lceil\frac{k+1}{2}\bigg\rceil.
    \end{align*} 
\end{proof}

\begin{definition} A \textbf{vertex} is said to be \textbf{K2} if it has at least two closed simple paths based at it. A \textbf{graph} is said to be \textbf{K2} if all of its vertices are K2. We define \textbf{K0} analogously. 
\end{definition}

Note that on vertices, this definition aligns with our definition of K1 given earlier. 

\begin{definition}
    Let $E$ be a graph. We say $E$ is \textbf{elementary} if $L(E)$ has no proper non-trivial ideals (i.e. $L(E)$ is a simple algebra).
\end{definition}

\begin{lem}
    Recall that we denote the hereditary saturated closure of a vertex $v$ as $T(v)$. If $E$ is an elementary graph, then $T(v) = E^0 \frall v\in E^0$. 
\end{lem}
\begin{proof}
    Suppose towards contradiction that a vertex $v$ on an elementary graph $E$ has $T(v)\neq E^0$, namely, there exists a vertex $w\in E^0 \setminus T(v)$. Then $\langle T(v) \rangle$ is a graded ideal, as discussed in Corollary 3.2, but $\langle T(v) \rangle \neq L(E)$ because $w \notin \langle T(v) \rangle$, contradicting the assumption that $E$ is elementary. Therefore, such a $w$ cannot exist, and $T(v) = E^0$.
\end{proof}

\begin{lem}
    Adding an edge to an elementary K2 graph yields another elementary K2 graph. 
\end{lem}
\begin{proof}
    Let $E$ be an elementary K2 graph and $L(E)$ its Leavitt path algebra. By Lemma 1, we know that $T(v) = E^0$ $\frall v\in E^0$. Hence, it follows that adding edges cannot change the structure of the graded ideals of $E$. That is, $T(v) = E^0$ cannot be made a proper subset of $E^0$ through the addition of edges. Furthermore, every vertex in $E$ is K2, so adding edges cannot introduce any non-graded ideals, since doing so would require that a vertex goes from having at least two closed simple paths to one or none.
\end{proof}

\begin{lem}
    Adding a loop to a K2 vertex has no effect on the ideal lattice.
\end{lem}
\begin{proof}
    Let $v$ be a K2 vertex in a graph $E$. For any hereditary saturated subset $S$, which generates a graded ideal, either $v\in S$ or $v\notin S$. If $v\notin S$, $S$ is still hereditary and saturated because there are no new edges added from $S$ to other vertices, and any vertex that only emits edges into $S$ already did so before the loop was added. If $v\in S$, then $S$ is still a hereditary saturated subset as well after the loop is added, by similar reasoning. \\
    On the other hand, any subset of vertices $S$ that was not originally hereditary and saturated can also either contain $v$ or not. If it doesn't contain $v$, then the only way it might become hereditary and saturated with the added loop is if $v$ emits edges only into $S$. However, since $v$ is K2, the existence of its closed simple paths implies that $v\in S$ already by the hereditary property, creating a contradiction with the assumption. If $v\in S$, then $S$ will also still not be hereditary and saturated when the loop is added. \\
    Hence, adding a loop to $v$ does not affect which sets are both hereditary and saturated, and therefore leaves the lattice of graded ideals unchanged. The loop also does not affect non-graded ideals because it does not create any cycles without exits: both closed simple paths based at $v$ have the other as an exit. Therefore, the loop does not affect the ideal lattice of the graph's Leavitt path algebra.
\end{proof}

\newpage

\begin{thm}
    For any Leavitt path algebra of a graph with two vertices, there are nine classes of ideal lattices of $\lambda$-reducible ideals, up to isomorphism. The classes have the following forms:

    \begin{figure}[H]
        \centering
        \includegraphics[width=0.8\linewidth]{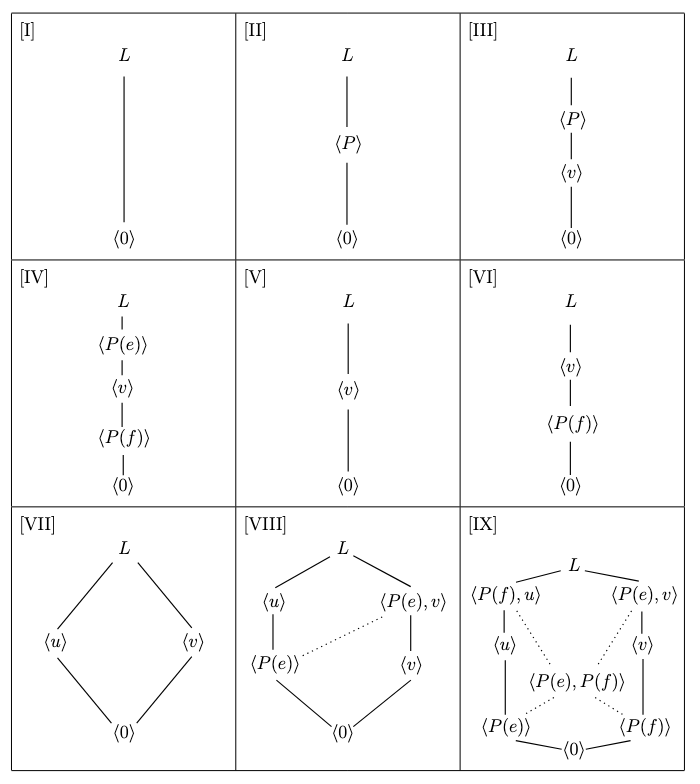}
        \label{fig3}
    \end{figure}
    
\end{thm}

\begin{remark}
    Here, $\langle u \rangle$ and $\langle v \rangle$ represent ideals generated by single vertices, while $L$ and $\langle 0 \rangle$ represent the entire algebra and the trivial ideal. On the other hand, $\langle P(e) \rangle$, $\langle P(f) \rangle$, and $\langle P \rangle$ indicate not singular ideals but rather families of non-graded ideals generated by polynomials of loop $e$, loop $f$, and the only loop present in a graph, respectively. Lastly, solid lines indicate containment, with the visually lower ideal being contained in the higher one, whereas dotted lines represent partial containment of one family of ideals in another family of ideals (for example, in lattice [VIII], some ideals in $\langle P(e) \rangle$ are contained in some ideals in $\langle P(e), v \rangle$). 
\end{remark}

\begin{proof} 
Let us consider graphs $E$ with $E^0 = \{ u,v \}$. By Lemma 3.12, we begin by noticing that, in terms of $\lambda$-reducible ideal lattice structure, a vertex with two loops is identical to a vertex with any number of loops greater than two. Hence, in terms of loops, we see that we only need to consider vertices having exactly 0, 1, or 2 loops. We can narrow down the relevant graphs by fixing the number of loops on each vertex and counting graphs based on the number of edges between the vertices. Furthermore, in the process of fixing the number of loops, we can use symmetry to force the number of loops on $u$ to be greater than or equal to the number of loops on $v$. This yields the cases (0,0), (1,0), (1,1), (2,0), (2,1) (2,2), where we have number of loops on $u$ in the first entry and number of loops on $v$ in the second. Now, because adding loops simply yields to another of these cases, we only need to consider the edges between them. 

First, notice that with respect to $\lambda$-reducible ideal lattice structure, having one edge pointing in one direction is the same as having arbitrarily many edges pointing only in that direction, since neither hereditary nor saturated subsets are affected by the addition of more edges in one direction. Furthermore, notice that for cases (0,0), (1,1), (2,2), the directed edges $uv$ and $vu$ yield identical graphs, so we only need to consider one of the two graphs. \\

For (0,0), we have graphs [1] (2 disjoint vertices), [2] (2-line), [3] ($uv$, $vu$), [4] ($uv, uv, vu$). \\

Notice that [4] is elementary K2, so we no longer need to consider the case where there are three or more edges that aren't all in the same direction by Lemma 2. Therefore, we only need to consider no edges between, $uv$, $vu$, and both $uv$ and $vu$. \\

Going on to (1,0), we have [5] (no edge), [6] ($uv$), [7] ($vu$) [8] ($uv, vu$). \\

Notice that [8] is also an elementary K2 graph, so for the remaining cases, because there is at least one loop at one vertex, there is no longer any need to consider the cases with both $uv$ and $vu$ by Lemma 2. Hence, for the remainder of the cases, there is only a need to consider no edges, $uv$, and $vu$. \\

Now we address (1,1). This is symmetrical, so we only add [9] (no edges) and [10] ($uv$). \\

For (2,0), we have [11] (none), [12] ($uv$), [13] ($vu$). \\

For (2,1), adding no edges yields the same lattice as graph [5], so we only have [14] ($uv$), [15] ($vu$). \\

For (2,2), none is identical to [1], so we only have [16] ($uv$). \\

All of the graphs listed above are in the following table:

\begin{figure}[H]
    \centering
    \includegraphics[width=0.9\linewidth]{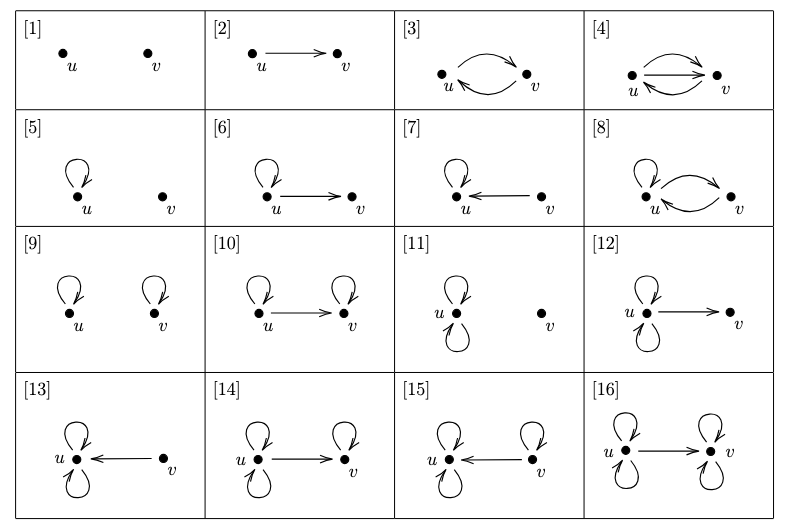}
    \caption{Types of 2-vertex directed graphs}
    \label{fig4}
\end{figure}

We observe that adding any edges to any of these graphs yields an equivalent case, hence we need only to consider the $\lambda$-reducible ideal lattices of these 16 graphs, and this is equivalent to considering hereditary saturated sets of vertices along with polynomials of cycles based at vertices not in our hereditary saturated sets. From this, it follows that the $\lambda$-reducible ideal lattices of the 16 graphs above are: \\

\begin{align*}
    [\text{I}] &\text{ for graphs [2],[4],[8], and [13]}\\
    [\text{II}] &\text{ for graphs [3] and [7]}\\
    [\text{III}] &\text{ for graphs [6] and [15]} \\
    [\text{IV}] &\text{ for graph [10]}\\
    [\text{V}] &\text{ for graphs [12] and [16]}\\
    [\text{VI}] &\text{ for graph [14]}\\
    [\text{VII}] &\text{ for graphs [1] and [11]}\\
    [\text{VIII}] &\text{ for graph [5]}\\
    [\text{IX}] &\text{ for graph [7]}.
\end{align*} 
\end{proof}

\section{Future Directions}

We now present some of the further questions that we did not investigate during this project. The first and most obvious of these questions is how the classification of ideals extends beyond just the $\lambda$-reducible ones. Is there such a nice classification in general? Are there graphs besides those satisfying Condition (K) with only $\lambda$-reducible ideals?

A second natural extension of our work is to try to produce analogous results for Leavitt path algebras with more than two vertices. Both symmetry and simplicity make the case of digraphs with two vertices a lot more manageable than even digraphs with three vertices, but perhaps there are different techniques that can be used.

Stepping further beyond the scope of what we studied, we also found ourselves wondering if there was a way to study the effects of turning or attaching edges on the Leavitt path algebra of a given digraph. There is the obvious change in the generators/relations, but the question is if there is a nicer characterization of the effect of such maneuvers that relates more directly to the Leavitt path algebra of the graph prior to the manipulation. If so, then we may be able to build the Leavitt path algebras of larger graphs out of those of smaller graphs. 

Finally, we also wondered if Morita Equivalence of two Leavitt path algebras would inform us of any similarities in their algebraic properties:

\begin{definition}
    An \textbf{equivalence of categories} consists of functors $F : \mathcal{C} \to \mathcal{D}$, $G : \mathcal{D} \to \mathcal{C}$ such that the composite functors $FG : \mathcal{D} \to \mathcal{D}$ and $GF : \mathcal{C} \to \mathcal{C}$ are both naturally isomorphic to the appropriate identity functors. In such a case, we say the categories $\mathcal{C}$ and $\mathcal{D}$ are \textbf{equivalent}. 
\end{definition}

Given a $K$-algebra $A$, the category of left modules over $A$, denoted $A\text{-}\mathsf{Mod}$, has as its objects left $A$-modules and morphisms left $A$-module homomorphisms. 

\begin{definition}
    Two $K$-algebras $A$ and $B$, are said to be \textbf{Morita equivalent} if the categories $A\text{-}\mathsf{Mod}$ and $B\text{-}\mathsf{Mod}$ are equivalent. 
\end{definition}

The connections between Morita equivalence and Leavitt path algebras are a subject of active study \cite{bock2024morita}\cite{clark2017subsets}\cite{kocc2022classification}.

\section{Acknowledgments} 

First, we thank Professors Ivan Contreras and Daniel van Wyk for mentoring the REU in which this research took place and guiding us through the writing process of this paper. The authors also thank our research partners Alex Kupersmith, Emma Keenan, and Ephrata Getachew, with whom we collaborated on the project. Finally, we thank the Amherst College Summer Undergraduate Research Fellowship (SURF) Program for organizing the REU and providing us with the opportunity to participate. 

\bibliographystyle{plain}
\bibliography{refs}
\nocite{*}

\end{document}